\documentclass{amsart}
\usepackage{amsmath,amssymb,amsthm,mathtools}
\usepackage{tikz}
\usetikzlibrary{arrows.meta,calc}
\usepackage{subcaption}
\usepackage{hyperref}

\theoremstyle{plain}
\newtheorem{theorem}{Theorem}[section]
\newtheorem{lemma}[theorem]{Lemma}
\newtheorem{corollary}[theorem]{Corollary}
\numberwithin{equation}{section}
\newtheorem{problem}[theorem]{Problem}

\newcommand{\Z}{\mathbb Z}
\newcommand{\W}{\mathbb W}
\DeclareMathOperator{\GL}{GL}
\DeclareMathOperator{\im}{im}

\newcommand{\br}[3]{[#1,#2,#3]}
\newcommand{\spn}[1]{\operatorname{span}_{\Z_2}\{#1\}}
\newcommand{\wideand}{\quad \text{ and } \quad}
\newcommand{\dual}{\ast}

\definecolor{pathA}{RGB}{29,84,131}
\definecolor{pathB}{RGB}{163,69,39}
\tikzset{
  edge/.style={line width=.65pt},
  boundary/.style={line width=.95pt},
  contractA/.style={draw=pathA,line width=2pt},
  contractB/.style={draw=pathB,line width=2pt,dash pattern=on 4pt off 1.8pt},
  dot/.style={circle,fill=black,inner sep=1.45pt},
  lab/.style={font=\small,inner sep=1.5pt},
  arrow/.style={-{Stealth[length=2.1mm]},line width=.7pt},
  faint/.style={draw=black!40,densely dashed,line width=.55pt}
}

\def\markpaths{0}
\newcommand{\fivepatch}[8]{%
    \coordinate (x0) at (0,1.9);
    \coordinate (x1) at (-1.65,.95);
    \coordinate (x2) at (-1.65,-.95);
    \coordinate (x3) at (0,-1.9);
    \coordinate (x4) at (1.65,-.95);
    \coordinate (x5) at (1.65,.95);
    \coordinate (u) at (-.62,0);
    \coordinate (v) at (.62,0);
    \draw[boundary] (x0)--(x1)--(x2)--(x3)--(x4)--(x5)--cycle;
    \foreach \j in {0,1,2,3} \draw[edge] (u)--(x\j);
    \foreach \j in {0,3,4,5} \draw[edge] (v)--(x\j);
    \draw[edge] (u)--(v);
    \ifnum\markpaths=1
        \draw[contractA] (x0)--(u)--(x2);
        \draw[contractB] (x3)--(v)--(x5);
    \fi
    \foreach \z in {x0,x1,x2,x3,x4,x5,u,v} \node[dot] at (\z) {};
    \node[lab,left=4pt] at (u) {$#1$};
    \node[lab,right=4pt] at (v) {$#2$};
    \node[lab,above=3pt] at (x0) {$#3$};
    \node[lab,above left=3pt] at (x1) {$#4$};
    \node[lab,below left=3pt] at (x2) {$#5$};
    \node[lab,below=3pt] at (x3) {$#6$};
    \node[lab,below right=3pt] at (x4) {$#7$};
    \node[lab,above right=3pt] at (x5) {$#8$};
}

\newcommand{\sixpatch}[9]{%
    \coordinate (x0) at (0,1.9);
    \coordinate (x1) at (-1.65,.85);
    \coordinate (x2) at (-1.65,-.85);
    \coordinate (x3) at (0,-1.9);
    \coordinate (x4) at (1.55,-1.30);
    \coordinate (x5) at (2.0,0);
    \coordinate (x6) at (1.55,1.30);
    \coordinate (u) at (-.60,0);
    \coordinate (v) at (.70,0);
    \fill[pathA!9] (u)--(x1)--(x2)--cycle;
    \draw[boundary] (x0)--(x1)--(x2)--(x3)--(x4)--(x5)--(x6)--cycle;
    \foreach \j in {0,1,2,3} \draw[edge] (u)--(x\j);
    \foreach \j in {0,3,4,5,6} \draw[edge] (v)--(x\j);
    \draw[edge] (u)--(v);
    \ifnum\markpaths=1
        \draw[contractA] (x0)--(u)--(x3);
        \draw[contractB] (x4)--(v)--(x6);
    \fi
    \foreach \z in {x0,x1,x2,x3,x4,x5,x6,u,v} \node[dot] at (\z) {};
    \node[lab,left=4pt] at (u) {$#1$};
    \node[lab,above=4pt] at (v) {$#2$};
    \node[lab,above=3pt] at (x0) {$#3$};
    \node[lab,above left=3pt] at (x1) {$#4$};
    \node[lab,below left=3pt] at (x2) {$#5$};
    \node[lab,below=3pt] at (x3) {$#6$};
    \node[lab,below right=3pt] at (x4) {$#7$};
    \node[lab,right=3pt] at (x5) {$#8$};
    \node[lab,above right=3pt] at (x6) {$#9$};
}

\definecolor{accent}{RGB}{32,76,111}
\tikzset{
    edge/.style={line width=.65pt},
    added/.style={draw=accent,line width=.8pt},
    dot/.style={circle,fill=black,inner sep=1.4pt},
    newdot/.style={circle,fill=accent,inner sep=1.5pt},
    lab/.style={font=\small,inner sep=1.5pt},
    arr/.style={-{Stealth[length=2mm]},line width=.6pt}
}
\newcommand{\subdivisionpicture}[1]{%
    \begin{tikzpicture}[scale=.80]
        \path[use as bounding box] (-1.66,-.31) rectangle (1.66,2.85);
        \coordinate (A) at (-1.3,0);
        \coordinate (B) at (0,2.4);
        \coordinate (C) at (1.3,0);
        \coordinate (X) at (0,.74);
        \coordinate (Y) at (-.48,1.03);
        \coordinate (Z) at (.48,1.03);
        \ifnum#1=0
            \fill[accent!6] (A)--(B)--(C)--cycle;
        \fi
        \ifnum#1=1
            \fill[accent!9] (A)--(B)--(X)--cycle;
        \fi
        \ifnum#1=2
            \fill[accent!9] (B)--(C)--(X)--cycle;
        \fi
        \draw[edge] (A)--(B)--(C)--cycle;
        \ifnum#1>0
            \draw[added] (A)--(X)--(B) (X)--(C);
        \fi
        \ifnum#1>1
            \draw[added] (A)--(Y)--(B) (Y)--(X);
        \fi
        \ifnum#1>2
            \draw[added] (B)--(Z)--(C) (Z)--(X);
        \fi
        \foreach \v/\name/\pos in {A/a/below left,B/b/above,C/c/below right}{%
        \node[dot] at (\v) {};
        \node[lab,\pos] at (\v) {$\name$};
        }
        \ifnum#1>0
            \node[newdot] at (X) {};
            \node[lab,below=3pt,text=accent] at (X) {$x$};
        \fi
        \ifnum#1>1
            \node[newdot] at (Y) {};
            \node[lab,left=3pt,text=accent] at (Y) {$y$};
        \fi
        \ifnum#1>2
            \node[newdot] at (Z) {};
            \node[lab,right=3pt,text=accent] at (Z) {$z$};
        \fi
    \end{tikzpicture}%
}

\title[Small covers and the Four Color Theorem]{Small covers and the Four Color Theorem}

\author[S. Choi]{Suyoung Choi}
\address{Department of Mathematics, Ajou University, 206, World Cup-ro, Yeongtong-gu, Suwon 16499, Republic of Korea}
\email{schoi@ajou.ac.kr}

\author[R. Kim]{Ringi Kim}
\address{Department of Mathematics, Inha University, 100, Inha-ro, Michuholgu, Incheon 22212, Republic of Korea}
\email{ringikim@inha.ac.kr}

\date{\today}
\thanks{The first named author was supported by the National Research Foundation of Korea Grant funded by the Korean Government (RS-2025-00521982).
The second named author was  supported by the National Research Foundation of Korea Grant funded by the Korean Government (RS-2025-00561867) and by an INHA UNIVERSITY Research Grant.}

\begin{document}

\begin{abstract}
We give a combinatorial proof that every simple convex $3$-polytope admits a small cover without using the Four Color Theorem.
We also prove that the assertion that every simple plane triangulation admits a mod~$2$ characteristic function with at most five values is equivalent to the Four Color Theorem.
\end{abstract}

\maketitle

\section{Introduction}\label{sec:intro}

We write $\Z_2$ for both the field with two elements and the cyclic group of order two.
Throughout the paper, plane graphs are drawn on $S^2$.
A \emph{plane triangulation} is a finite connected loopless plane multigraph in which the boundary of every face is a cycle of length three.
Parallel edges are allowed, but each face has three distinct vertices.
The graph $K_3$ with its two faces is included.
A plane triangulation is called \emph{simple} if it has no parallel edges.

Put
$$
    \W=\Z_2^3,\quad \W^\times=\W\setminus\{0\}, \wideand \br{a}{b}{c}=\det(a,b,c)\in\Z_2.
$$
A \emph{mod~$2$ characteristic function} on a plane triangulation $G$ is a map $\lambda\colon V(G)\to\W^\times$ such that
\begin{equation}\label{eq:char}
    \br{\lambda(x)}{\lambda(y)}{\lambda(z)}=1 \quad\text{for every facial triangle }xyz.
\end{equation}
Thus the labels on each facial triangle form a basis of $\W$.
The condition is imposed on facial triangles, not on all $3$-cycles.
In particular, adjacent vertices have distinct labels.
We use \emph{characteristic function} for short.

For $r\in\{4,5,6,7\}$, let $E_r$ denote the following assertion.
\begin{quote}
    \noindent\textup{$E_r$:} Every simple plane triangulation admits a mod~$2$ characteristic function with at most $r$ values.
\end{quote}
The implications
$$
    E_4\ \Longrightarrow\ E_5\ \Longrightarrow\ E_6\ \Longrightarrow\ E_7
$$
are immediate.
Moreover, $E_4$ is equivalent to the Four Color Theorem.
Indeed, a proper coloring with at most four colors gives a characteristic function by replacing the colors with
$$
    e_1,\quad e_2,\quad e_3,\wideand e_1+e_2+e_3,
$$
where $e_1,e_2,e_3$ form the standard basis of $\W$.
Any three of these four vectors form a basis of $\W$.
Conversely, a characteristic function with at most four values is a proper coloring with at most four colors.
Every finite simple planar graph is a subgraph of a simple plane triangulation, so $E_4$ gives the Four Color Theorem.

We prove the following two results.
The proof of the first does not use the Four Color Theorem.

\begin{theorem}\label{thm:graph}
    Every plane triangulation admits a mod~$2$ characteristic function.
\end{theorem}

The second result shows that allowing a fifth value does not weaken this universal existence assertion.

\begin{theorem}\label{thm:five-four}
    The Four Color Theorem is equivalent to the assertion that every simple plane triangulation admits a mod~$2$ characteristic function with at most five values.
\end{theorem}

The assertion~$E_7$ is related to the existence of small covers.
A \emph{small cover} over a simple convex $n$-polytope $P$ is a closed $n$-manifold with a locally standard action of~$\Z_2^n$ whose orbit space is $P$.
Davis and Januszkiewicz~\cite{Davis-Januszkiewicz1991} introduced small covers and constructed them from mod~$2$ characteristic functions.
For $n=3$, the facets of $P$ correspond to the vertices of the simple plane triangulation dual to~$\partial P$.
The three facets meeting at a vertex of $P$ correspond to a facial triangle.
A characteristic function on this triangulation therefore gives a small cover over $P$.

Three-dimensional small covers have received considerable attention because they connect the combinatorics of simple polytopes with the topology and geometry of $3$-manifolds~\cite{Izmestiev2001,Kuroki2010,Lu-Yu2011,Nishimura2012,Cao-Lu2011,Grujic2022, Gorchakov2024}.
In particular, the orientable case is closely related to four-colorings and has been studied in~\cite{Nakayama-Nishimura2005, Erokhovets2022, Choi-Jang-Yoon2026}.

Davis and Januszkiewicz~\cite[Example~1.21]{Davis-Januszkiewicz1991} used the Four Color Theorem to observe that every simple convex $3$-polytope admits a small cover.
This turns a coloring theorem for planar graphs into an existence theorem for closed $3$-manifolds with prescribed orbit polytopes.
Our proof of Theorem~\ref{thm:graph} establishes this existence result without using the Four Color Theorem.
We note that such an attempt was previously made in~\cite{Liu2013}.
However, that proof contains a gap in a key reduction step, as explained in Appendix~\ref{app:Liu}.

Our proof of Theorem~\ref{thm:graph} instead works directly with plane triangulations.
The proof uses local reductions at adjacent vertices of degrees~$(5,5)$ or~$(5,6)$.
In each case, every characteristic function on the reduced triangulation allows the labels to be extended across the original configuration.
Parallel edges are allowed because they may arise during these reductions.
See Sections~\ref{sec:reductions} and~\ref{sec:light}.

To prove Theorem~\ref{thm:five-four}, we use stellar subdivisions to make all vertex degrees odd.
We apply $E_5$ to the enlarged triangulation and show by a parity argument that the labels lie in an affine plane of four nonzero vectors.
Restricting to the original vertices then proves $E_4$.
See Section~\ref{sec:five-four}.

Our results suggest the following natural question.
\begin{problem}
    Can one prove, without using the Four Color Theorem, that every simple plane triangulation admits a mod~$2$ characteristic function with at most six values?
    Alternatively, can one deduce the Four Color Theorem directly from this assertion?
\end{problem}

\section{Reductions of plane triangulations}\label{sec:reductions}

All vector spaces and determinants in this section are over $\Z_2$.
In particular, the order of the three arguments of a determinant does not affect its value.

\begin{lemma}\label{lem:shortcycle}
    Let $a_0,\ldots,a_{k-1}\in\W^\times$, where $k=3$ or $4$, and suppose that $a_i\ne a_{i+1}$, with indices taken modulo~$k$.  
    Then there is an element~$x\in\W^\times$ such that
    $$
        \br{x}{a_i}{a_{i+1}}=1 \qquad\text{for all }i.
    $$
\end{lemma}
\begin{proof}
    For $k=3$, the vectors $a_0$, $a_1$, and $a_2$ are distinct.  
    If they form a basis, take $x=a_0+a_1+a_2$.  
    Otherwise they are the three nonzero vectors in a plane, and any vector outside that plane can be used as~$x$.
    
    For $k=4$, write the cyclic sequence as $a,b,c,d$.
    If these vectors lie in a plane, take $x$ outside that plane.
    Otherwise, after a cyclic relabeling, we may assume that $a,b,c$ form a basis.
    If $d=a+c$, take $x=a+b+c$.
    Otherwise, $d\notin\spn{a,c}$, since $d\ne a,c$, and we take $x=a+c$.
    In both cases,
    $$
        [x,a,b]=[x,b,c]=[x,c,d]=[x,d,a]=1.
    $$
\end{proof}

See Figure~\ref{fig:shortcycle} for illustrations.

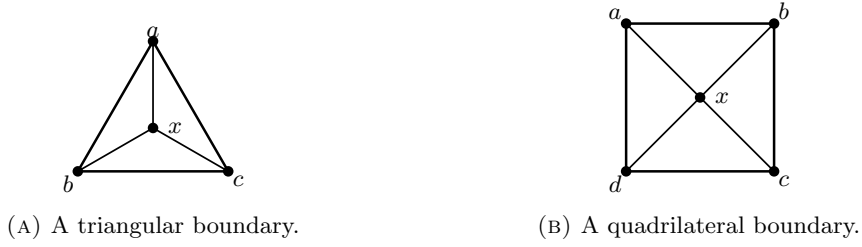
\begin{figure}
    \centering
    \begin{subfigure}{.43\textwidth}\centering
        \begin{tikzpicture}[scale=.85]
            \coordinate (a) at (90:1.35); \coordinate (b) at (210:1.35);
            \coordinate (c) at (330:1.35); \coordinate (x) at (0,0);
            \draw[boundary] (a)--(b)--(c)--cycle;
            \foreach \v in {a,b,c}{\draw[edge] (x)--(\v);\node[dot] at (\v) {};}
            \node[dot] at (x) {};
            \node[lab,above] at (a) {$a$};\node[lab,below left] at (b) {$b$};
            \node[lab,below right] at (c) {$c$};\node[lab,right=4pt] at (x) {$x$};
        \end{tikzpicture}
        \caption{A triangular boundary.}
    \end{subfigure}\hfill
    \begin{subfigure}{.43\textwidth}\centering
        \begin{tikzpicture}[scale=.85]
            \coordinate (a) at (-1.15,1.15);\coordinate (b) at (1.15,1.15);
            \coordinate (c) at (1.15,-1.15);\coordinate (d) at (-1.15,-1.15);
            \coordinate (x) at (0,0);
            \draw[boundary] (a)--(b)--(c)--(d)--cycle;
            \foreach \v in {a,b,c,d}{\draw[edge] (x)--(\v);\node[dot] at (\v) {};}
            \node[dot] at (x) {};
            \node[lab,above left] at (a) {$a$};\node[lab,above right] at (b) {$b$};
            \node[lab,below right] at (c) {$c$};\node[lab,below left] at (d) {$d$};
            \node[lab,right=4pt] at (x) {$x$};
        \end{tikzpicture}
        \caption{A quadrilateral boundary.}
    \end{subfigure}
    \caption{For a triangular or quadrilateral boundary, a central label can be added.}
    \label{fig:shortcycle}
\end{figure}

\begin{corollary}\label{cor:path}
    If $p,a,d,e\in\W^\times$ satisfy $p\ne a$, $a\ne d$, and $d\ne e$, then there is an element $q\in\W^\times$ such that
    $$
        \br{q}{p}{a}=\br{q}{a}{d}=\br{q}{d}{e}=1.
    $$
\end{corollary}
\begin{proof}
    If $p\ne e$, apply Lemma~\ref{lem:shortcycle} with $k=4$ and $(a_0,a_1,a_2,a_3)=(p,a,d,e)$.
    If $p=e$, apply it with $k=3$ and $(a_0,a_1,a_2)=(p,a,d)$.
\end{proof}

A \emph{$k$-cycle} is a cycle of length $k$.
In particular, two distinct parallel edges form a $2$-cycle.
By the Jordan curve theorem, this cycle separates $S^2$ into two closed discs~$D_1$ and~$D_2$.
In each $D_i$, identify the two boundary edges by a homeomorphism that fixes their endpoints.  
The quotient is again a sphere, and we denote the induced plane graph by $G_i$.
Its faces are precisely the triangular faces that lay in $D_i$ before the identification, and their three vertices remain distinct.

\begin{lemma}\label{lem:split2}
    Let $G$ be a loopless plane triangulation with a $2$-cycle on vertices $a$ and $b$, 
    and let $G_1$ and $G_2$ be the two plane triangulations just described.
    Each $G_i$ has fewer vertices than $G$.  
    If both admit characteristic functions, then so does $G$.
\end{lemma}
\begin{proof}
    Each disc contains a vertex other than $a$ and $b$, 
    since it is triangulated by faces with three distinct vertices.
    Thus $|V(G_i)|<|V(G)|$.
    In each characteristic function, the labels at $a$ and $b$ are independent.
    An element of $\GL(3,\Z_2)$ sends either ordered independent pair to the other.  
    Apply such an element to one function and glue the two functions along $a$ and $b$.
\end{proof}

For a simple plane triangulation, a \emph{separating triangle} is a $3$-cycle with vertices on both sides.  
A $3$-cycle that does not bound a face is separating.  
Cutting along a separating triangle and adding a triangular face on each side gives two smaller plane triangulations.

\begin{lemma}\label{lem:split3}
    If the two plane triangulations obtained by splitting a simple plane triangulation at a separating triangle admit characteristic functions, then the original triangulation does as well.
\end{lemma}
\begin{proof}
    The labels on each added face form an ordered basis.  
    Match these bases by an element of $\GL(3,\Z_2)$ and glue the functions along the triangle.
\end{proof}

\begin{lemma}\label{lem:low}
    Let $v$ be a vertex of degree $3$ or $4$ in a simple plane triangulation $G$.  
    Delete $v$ and triangulate the resulting face, adding one diagonal in the degree-four case.  
    Denote the resulting plane triangulation by~$G'$.  
    Every characteristic function on~$G'$ extends to a characteristic function on $G$.
\end{lemma}
\begin{proof}
    The neighbors of $v$ occur on a cycle of length three or four.
    Consecutive boundary labels are distinct, so Lemma~\ref{lem:shortcycle} supplies the label at $v$.
    If the new diagonal has the same endpoints as an edge outside the face, it is kept as a parallel edge.  
    For $G=K_4$, deleting a vertex leaves $K_3$ with its two faces, which is allowed.
\end{proof}

\section{Local Reductions}\label{sec:light}

We use the following classical result of Wernicke~\cite{Wernicke1904}.
\begin{lemma}\label{lem:light}
    Every simple plane triangulation with minimum degree at least five has an edge whose endpoint degrees are $(5,5)$ or $(5,6)$.
\end{lemma}

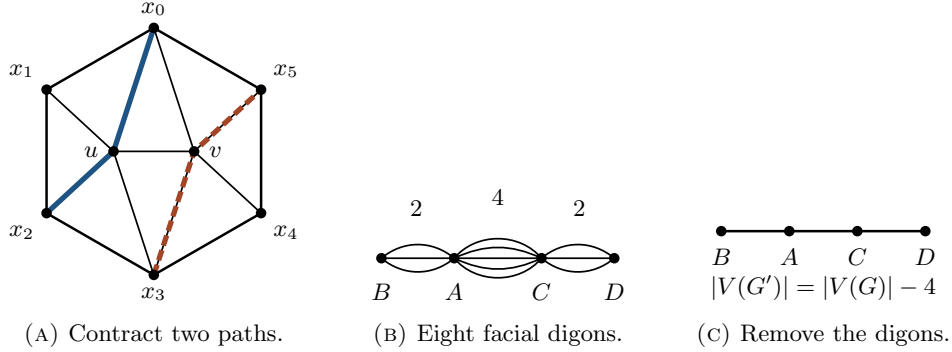
\begin{figure}
    \centering
    \begin{subfigure}{.32\textwidth}\centering
        \begin{tikzpicture}[scale=.86]
            \def\markpaths{1}
            \fivepatch{u}{v}{x_0}{x_1}{x_2}{x_3}{x_4}{x_5}
        \end{tikzpicture}
        \caption{Contract two paths.}
    \end{subfigure}\hfill
    \begin{subfigure}{.33\textwidth}\centering
        \begin{tikzpicture}[scale=.77]
            \coordinate (B) at (-2,0);\coordinate (A) at (-.75,0);
            \coordinate (C) at (.75,0);\coordinate (D) at (2,0);
            \foreach \a/\b in {B/A,C/D}{
                \draw[edge] (\a)--(\b);
                \draw[edge] (\a) to[bend left=40] (\b);
                \draw[edge] (\a) to[bend right=40] (\b);
            }
            \draw[edge] (A)--(C);
            \foreach \bend in {25,50}{
                \draw[edge] (A) to[bend left=\bend] (C);
                \draw[edge] (A) to[bend right=\bend] (C);
            }
            \foreach \v in {B,A,C,D}{\node[dot] at (\v) {};\node[lab,below=8pt,fill=white] at (\v) {$\v$};}
            \node[lab] at (-1.4,.85) {$2$};\node[lab] at (0,1.05) {$4$};\node[lab] at (1.4,.85) {$2$};
        \end{tikzpicture}
        \caption{Eight facial digons.}
    \end{subfigure}\hfill
    \begin{subfigure}{.28\textwidth}\centering
        \begin{tikzpicture}[scale=.75]
            \foreach \x/\v in {-1.8/B,-.6/A,.6/C,1.8/D}{\coordinate (\v) at (\x,0);}
            \draw[boundary] (B)--(A)--(C)--(D);
            \foreach \v in {B,A,C,D}{\node[dot] at (\v) {};\node[lab,below=5pt] at (\v) {$\v$};}
            \node[lab] at (0,-1.0) {$|V(G')|=|V(G)|-4$};
        \end{tikzpicture}
        \caption{Remove the digons.}
    \end{subfigure}
    \caption{The $(5,5)$ reduction.}
    \label{fig:55reduction}
\end{figure}

For the two reductions below, assume that $G$ is a simple plane triangulation with no separating triangle and with minimum degree at least five.

If $u$ and $v$ are adjacent, they have exactly two common neighbors.  
Any other common neighbor would give a nonfacial $3$-cycle.
Consequently, the faces incident with $u$ or $v$ form a triangulated disc with a simple boundary in each case below.

Also, if two nonconsecutive neighbors of $u$ were adjacent, their edge together with $u$ would give a separating triangle. 

Suppose that $\deg(u)=\deg(v)=5$.  
Write the boundary of this disc cyclically as $x_0x_1x_2x_3x_4x_5x_0$, so that its eight triangular faces are
\begin{equation}\label{eq:55faces}
    \begin{gathered}
         ux_0x_1,\quad ux_1x_2,\quad ux_2x_3,\\
         vx_3x_4,\quad vx_4x_5,\quad vx_5x_0,\\
         uvx_0,\quad uvx_3
    \end{gathered}
\end{equation}
as in Figure~\ref{fig:55reduction}(a).

Contract the two paths
$$
     x_0-u-x_2 \wideand x_3-v-x_5.
$$
They are induced and vertex-disjoint.
No loops arise, since the contracted paths are induced.
Every face outside the disc has vertices in three distinct contraction classes, so it remains triangular.
Write $A=\{x_0,u,x_2\}$ and $C=\{x_3,v,x_5\}$ for the two contraction classes, and put $B=x_1$ and $D=x_4$.

All eight faces in \eqref{eq:55faces} become $2$-faces. 
More precisely, the local graph consists of three parallel edges between $B$ and $A$, five between $A$ and $C$, and three between~$C$ and~$D$.
There are respectively two, four, and two $2$-faces in these three bundles as in Figure~\ref{fig:55reduction}(b). 

To remove a $2$-face bounded by $\alpha$ and $\beta$, delete $\alpha$ and merge the $2$-face with the face on the other side of $\alpha$.  
If that other face is a triangle, its side $\alpha$ is replaced by $\beta$.
Its three vertices do not change.
If it is a $2$-face, the merged face is again a $2$-face.  
This operation changes the cell decomposition, not the underlying sphere.
See Figure~\ref{fig:digon}.

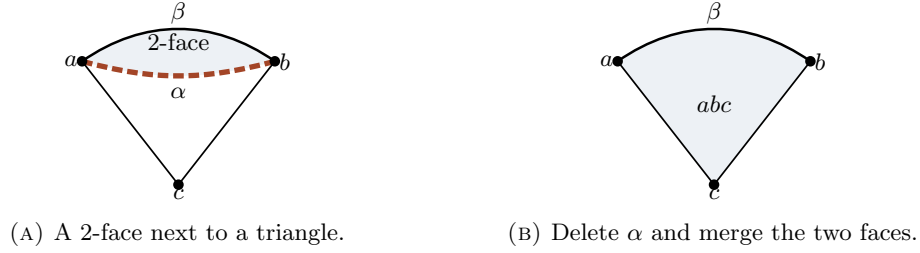
\begin{figure}
    \centering
    \begin{subfigure}{.44\textwidth}\centering
        \begin{tikzpicture}[scale=1.02]
            \coordinate (a) at (-1.25,0);
            \coordinate (b) at (1.25,0);
            \coordinate (c) at (0,-1.6);
            \fill[pathA!8] (a) to[bend left=35] (b) to[bend left=15] (a);
            \draw[boundary] (a) to[bend left=35] node[lab,above] {$\beta$} (b);
            \draw[contractB] (a) to[bend right=15] node[lab,below=2pt] {$\alpha$} (b);
            \draw[edge] (a)--(c)--(b);
            \foreach \v in {a,b,c}\node[dot] at (\v) {};
            \node[lab,left] at (a) {$a$};\node[lab,right] at (b) {$b$};\node[lab,below] at (c) {$c$};
            \node[lab] at (0,.25) {$2$-face};
        \end{tikzpicture}
        \caption{A $2$-face next to a triangle.}
    \end{subfigure}\hfill
    \begin{subfigure}{.44\textwidth}\centering
        \begin{tikzpicture}[scale=1.02]
            \coordinate (a) at (-1.25,0);
            \coordinate (b) at (1.25,0);
            \coordinate (c) at (0,-1.6);
            \fill[pathA!8] (a) to[bend left=35] (b)--(c)--cycle;
            \draw[boundary] (a) to[bend left=35] node[lab,above] {$\beta$} (b);
            \draw[edge] (a)--(c)--(b);
            \foreach \v in {a,b,c}\node[dot] at (\v) {};
            \node[lab,left] at (a) {$a$};\node[lab,right] at (b) {$b$};\node[lab,below] at (c) {$c$};
            \node[lab] at (0,-.55) {$abc$};
        \end{tikzpicture}
        \caption{Delete $\alpha$ and merge the two faces.}
    \end{subfigure}
    \caption{Removing a facial digon.}
    \label{fig:digon}
\end{figure}

Remove these eight $2$-faces by the operation above.
All other parallel edges are retained.
The image of the local part is the path $B-A-C-D$ as in Figure~\ref{fig:55reduction}(c).
Then we obtain a plane triangulation $G'$ with four fewer vertices, and we call it the \emph{$(5,5)$ reduction} of~$G$.

\begin{lemma}\label{lem:55}
    Every mod~$2$ characteristic function on $G'$ obtained from $G$ by the $(5,5)$ reduction gives a mod~$2$ characteristic function on~$G$.
\end{lemma}
\begin{proof}
    Let $a,b,c,d \in \W^\times$ be the labels at $A$, $B$, $C$, and $D$, respectively.
    Since $B-A-C-D$ is a path, we have $a\ne b$, $a\ne c$, and $c\ne d$.
    After pullback, the labels on the boundary of the original disc are $(a,b,a,c,d,c)$ as in Figure~\ref{fig:55labels}.
    We need labels $p$ and $q$ at $u$ and $v$, respectively.
    
    After a change of basis, we may assume that $a=e_1$ and $c=e_2$.
    Set $p=(x,y,1)$ and $q=(x+1,y+1,1)=p+a+c$.
    Then
    \begin{equation}\label{eq:55auto}
        \br{p}{a}{c}=\br{q}{c}{a} =\br{p}{q}{a}=\br{p}{q}{c}=1.
    \end{equation}
    
    The only remaining conditions are
    \begin{equation}\label{eq:55linear}
        \br{p}{a}{b}=1 \wideand \br{q}{c}{d}=1.
    \end{equation}
    Write $b=(b_1,b_2,b_3)$ and $d=(d_1,d_2,d_3)$.
    Their left-hand sides are $b_2+yb_3$ and $d_1+(x+1)d_3$, respectively.  
    If $b_3=1$, take $y=1+b_2$.  
    If $b_3=0$, then $b\ne0,a$ implies $b_2=1$, and $y$ is arbitrary.
    If $d_3=1$, take $x=d_1$.  
    If $d_3=0$, then $d\ne0,c$ implies $d_1=1$, and $x$ is arbitrary.  
    The choices of $x$ and $y$ are independent, so both conditions in~\eqref{eq:55linear} hold.
    
    Thus the labels on all eight faces satisfy the characteristic condition by~\eqref{eq:55auto} and~\eqref{eq:55linear}.
    Every other face has the labels pulled back from a triangular face of~$G'$.  
    This proves the lemma.
\end{proof}
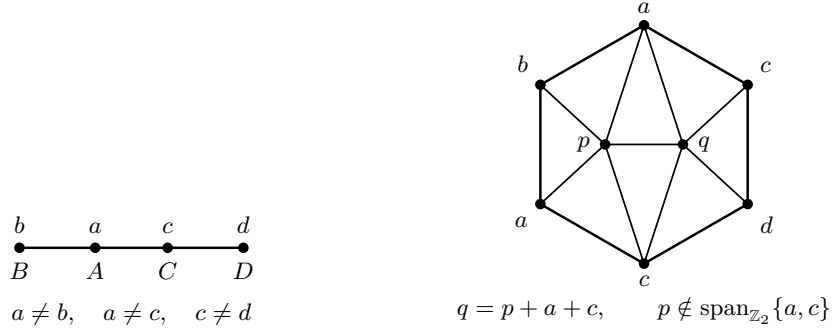
\begin{figure}
    \centering
    \begin{subfigure}{.40\textwidth}\centering
        \begin{tikzpicture}[scale=.87]
            \coordinate (B) at (-1.7,0);\coordinate (A) at (-.55,0);
            \coordinate (C) at (.55,0);\coordinate (D) at (1.7,0);
            \draw[boundary] (B)--(A)--(C)--(D);
            \foreach \v/\val in {B/b,A/a,C/c,D/d}{
                \node[dot] at (\v) {};\node[lab,above=4pt] at (\v) {$\val$};
                \node[lab,below=4pt] at (\v) {$\v$};
            }
            \node[lab] at (0,-1.0) {$a\ne b,\quad a\ne c,\quad c\ne d$};
        \end{tikzpicture}
    \end{subfigure}\hfill
    \begin{subfigure}{.53\textwidth}\centering
        \begin{tikzpicture}[scale=.83]
            \fivepatch{p}{q}{a}{b}{a}{c}{d}{c}
            \node[lab] at (0,-2.65) {$q=p+a+c,\qquad p\notin\spn{a,c}$};
        \end{tikzpicture}
    \end{subfigure}
    \caption{Restoration in the $(5,5)$ case.}
    \label{fig:55labels}
\end{figure}

Suppose that $\deg(u)=5$ and $\deg(v)=6$.  
Write the boundary of this disc cyclically as $x_0x_1x_2x_3x_4x_5x_6x_0$.
Its nine triangular faces are
\begin{equation}\label{eq:56faces}
    \begin{gathered}
        ux_0x_1,\quad ux_1x_2,\quad ux_2x_3,\\
        vx_3x_4,\quad vx_4x_5,\quad vx_5x_6,\quad vx_6x_0,\\
        uvx_0,\quad uvx_3.
    \end{gathered}
\end{equation}
Contract the two induced paths $x_0-u-x_3$ and $x_4-v-x_6$, and put
$$
    A=\{x_0,u,x_3\},\quad D=\{x_4,v,x_6\}, \quad B=x_1,\quad C=x_2,\wideand E=x_5.
$$
Eight faces become $2$-faces, while $ux_1x_2$ remains as the triangular face $ABC$.  
There is one $2$-face along each of $AB$ and $AC$, four between $A$ and $D$, and two between $D$ and $E$.
After deleting these $2$-faces, the image of the local part is the triangle $ABC$ together with the path $A-D-E$. 
See Figure~\ref{fig:56reduction}.
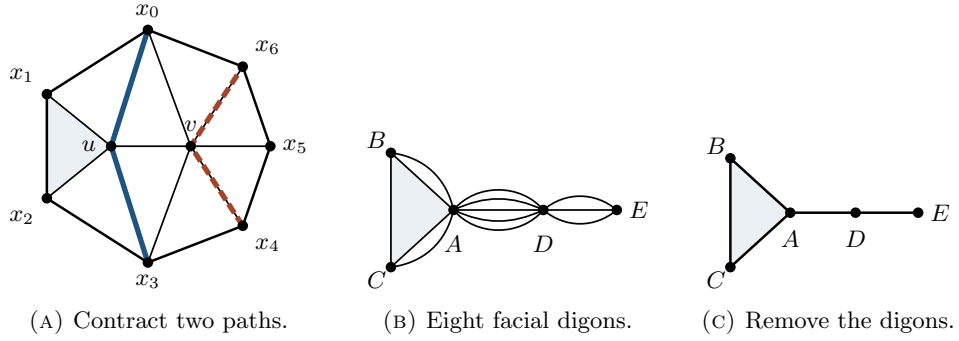
\begin{figure}
    \centering
    \begin{subfigure}{.33\textwidth}\centering
        \begin{tikzpicture}[scale=.81]
            \def\markpaths{1}
            \sixpatch{u}{v}{x_0}{x_1}{x_2}{x_3}{x_4}{x_5}{x_6}
        \end{tikzpicture}
        \caption{Contract two paths.}
    \end{subfigure}\hfill
    \begin{subfigure}{.34\textwidth}\centering
        \begin{tikzpicture}[scale=.72]
            \coordinate (A) at (-.9,0);\coordinate (B) at (-2.05,1.05);
            \coordinate (C) at (-2.05,-1.05);\coordinate (D) at (.75,0);\coordinate (E) at (2.1,0);
            \fill[pathA!9] (A)--(B)--(C)--cycle;
            \draw[edge] (A)--(B)--(C)--cycle;
            \draw[edge] (A) to[bend right=35] (B);
            \draw[edge] (A) to[bend left=35] (C);
            \draw[edge] (A)--(D);
            \foreach \bend in {25,50}{\draw[edge] (A) to[bend left=\bend] (D);\draw[edge] (A) to[bend right=\bend] (D);}
            \draw[edge] (D)--(E);
            \draw[edge] (D) to[bend left=40] (E);\draw[edge] (D) to[bend right=40] (E);
            \foreach \v in {A,B,C,D,E}\node[dot] at (\v) {};
            \node[lab,below=8pt] at (A) {$A$};
            \node[lab,above left] at (B) {$B$};
            \node[lab,below left] at (C) {$C$};
            \node[lab,below=8pt,fill=white] at (D) {$D$};
            \node[lab,right=3pt] at (E) {$E$};
        \end{tikzpicture}
        \caption{Eight facial digons.}
    \end{subfigure}\hfill
    \begin{subfigure}{.27\textwidth}\centering
        \begin{tikzpicture}[scale=.72]
            \coordinate (A) at (-.55,0);\coordinate (B) at (-1.65,1);
            \coordinate (C) at (-1.65,-1);\coordinate (D) at (.65,0);\coordinate (E) at (1.8,0);
            \fill[pathA!9] (A)--(B)--(C)--cycle;
            \draw[boundary] (A)--(B)--(C)--cycle;\draw[boundary] (A)--(D)--(E);
            \foreach \v in {A,B,C,D,E}\node[dot] at (\v) {};
            \node[lab,below=5pt] at (A) {$A$};\node[lab,above left] at (B) {$B$};
            \node[lab,below left] at (C) {$C$};\node[lab,below=5pt] at (D) {$D$};
            \node[lab,right=3pt] at (E) {$E$};
        \end{tikzpicture}
        \caption{Remove the digons.}
    \end{subfigure}
    \caption{The $(5,6)$ reduction. }
    \label{fig:56reduction}
\end{figure}
The resulting plane triangulation~$G'$ has four fewer vertices than $G$.
We call this the \emph{$(5,6)$ reduction} of $G$.

\begin{lemma}\label{lem:56}
    Every mod~$2$ characteristic function on $G'$ obtained from $G$ by the $(5,6)$ reduction gives a mod~$2$ characteristic function on $G$.
\end{lemma}
\begin{proof}
    Write $a$, $b$, $c$, $d$, and $e$ for the labels at $A$, $B$, $C$, $D$, and $E$, respectively.  
    The surviving face and path give $\br{a}{b}{c}=1$, $a\ne d$, and $d\ne e$.
    The labels on the original boundary are therefore $(a,b,c,a,d,e,d)$.
    Give $u$ the label $p=a+b+c$.
    Then 
    \begin{equation}\label{eq:56p}
        \br{p}{a}{b}=\br{p}{b}{c}=\br{p}{c}{a}=1. 
    \end{equation}
    Also $p\ne a$, since $a$, $b$, and $c$ form a basis.
    By applying Corollary~\ref{cor:path} to $p,a,d,e$, we obtain $q$ such that
    \begin{equation}\label{eq:56q}
        \br{q}{p}{a}=\br{q}{a}{d}=\br{q}{d}{e}=1.
    \end{equation}
    Give $v$ the label $q$.
    See Figure~\ref{fig:56labels}.

    All nine faces in \eqref{eq:56faces} satisfy the characteristic condition by \eqref{eq:56p} and~\eqref{eq:56q}.
    The other faces inherit it from $G'$.
\end{proof}
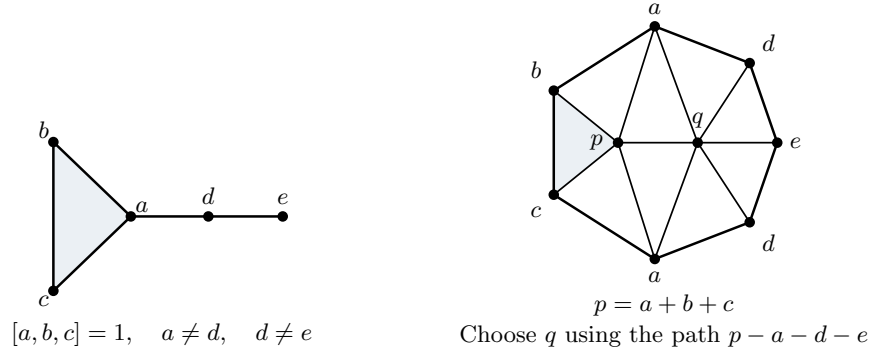
\begin{figure}
    \centering
    \begin{subfigure}{.41\textwidth}\centering
        \begin{tikzpicture}[scale=.82]
            \coordinate (A) at (-.45,0);\coordinate (B) at (-1.7,1.2);
            \coordinate (C) at (-1.7,-1.2);\coordinate (D) at (.8,0);\coordinate (E) at (2,0);
            \fill[pathA!9] (A)--(B)--(C)--cycle;
            \draw[boundary] (A)--(B)--(C)--cycle;\draw[boundary] (A)--(D)--(E);
            \foreach \v in {A,B,C,D,E}\node[dot] at (\v) {};
            \node[lab,above right] at (A) {$a$};\node[lab,above left] at (B) {$b$};
            \node[lab,below left] at (C) {$c$};\node[lab,above=3pt] at (D) {$d$};
            \node[lab,above=3pt] at (E) {$e$};
            \node[lab] at (.05,-1.9) {$[a,b,c]=1,\quad a\ne d,\quad d\ne e$};
        \end{tikzpicture}
    \end{subfigure}\hfill
    \begin{subfigure}{.54\textwidth}\centering
        \begin{tikzpicture}[scale=.81]
            \sixpatch{p}{q}{a}{b}{c}{a}{d}{e}{d}
            \node[lab] at (.15,-2.65) {$p=a+b+c$};
            \node[lab] at (.15,-3.15) {Choose $q$ using the path $p-a-d-e$};
        \end{tikzpicture}
    \end{subfigure}
    \caption{Restoration in the $(5,6)$ case.}
    \label{fig:56labels}
\end{figure}

Now, we are ready to prove our first theorem.

\begin{proof}[Proof of Theorem~\ref{thm:graph}]
    Suppose that a counterexample exists, and choose one, $G$, with as few vertices as possible.  
    The plane triangulation with three vertices is $K_3$ with its two faces.
    Assigning a basis to its vertices gives a mod~$2$ characteristic function.  
    Thus $|V(G)|\ge4$.
    
    By Lemma~\ref{lem:split2}, $G$ has no parallel edges, so it is simple.
    By Lemma~\ref{lem:split3}, it has no separating triangle.
    A simple plane triangulation with at least four vertices has minimum degree at least three.
    Lemma~\ref{lem:low} excludes vertices of degree three or four.  
    Hence $G$ has minimum degree at least five.
    
    Lemma~\ref{lem:light} gives an edge $uv$ with endpoint degrees $(5,5)$ or $(5,6)$.  
    Applying the corresponding reduction, we obtain a loopless plane triangulation~$G'$ with $|V(G)|-4$ vertices.  It need not be simple, but it belongs to the same class in which $G$ was chosen minimal.
    
    By minimality, $G'$ has a characteristic function.
    Lemma~\ref{lem:55} or Lemma~\ref{lem:56} restores a characteristic function on $G$.  
    This is a contradiction.
\end{proof}

\section{Five values and four colors}\label{sec:five-four}

We now prove Theorem~\ref{thm:five-four}.
All triangulations in this section are simple.
A \emph{stellar subdivision of a facial triangle} inserts a new vertex in the face and joins it to its three boundary vertices.  
This preserves simplicity and the triangulation, keeps every old edge, and adds no edge between old vertices.

\begin{lemma}\label{lem:parity}
    Every simple plane triangulation $G$ can be extended by stellar subdivisions of facial triangles to a simple plane triangulation $G'$ in which every vertex has odd degree.  
    Moreover, $G'[V(G)]=G$.
\end{lemma}
\begin{proof}
We first change the parity of the degree of one chosen vertex $b$, without changing the parity at any other existing vertex.  
Choose a facial triangle $abc$ containing $b$.  
Insert a vertex~$x$ in~$abc$, then a vertex~$y$ in~$abx$, and finally a vertex~$z$ in~$bcx$.
The face $bcx$ is unchanged by the second insertion.
See Figure~\ref{fig:subdivisions}.

The degrees of $a,b,c$ increase by $2,3,2$, respectively.
The other existing vertices are unchanged.
The new vertices satisfy $\deg(x)=5$ and $\deg(y)=\deg(z)=3$.
Thus only the parity at $b$ changes among the existing vertices.

Perform this operation once at each original vertex of even degree.
Every later operation preserves the parity at all existing vertices other than its chosen vertex.  
Since only original vertices are chosen, all vertices of the final triangulation $G'$ have odd degree.
Each insertion preserves all old edges and adds no edge between old vertices, so $G'[V(G)]=G$.
\end{proof}

\begin{figure}[h]
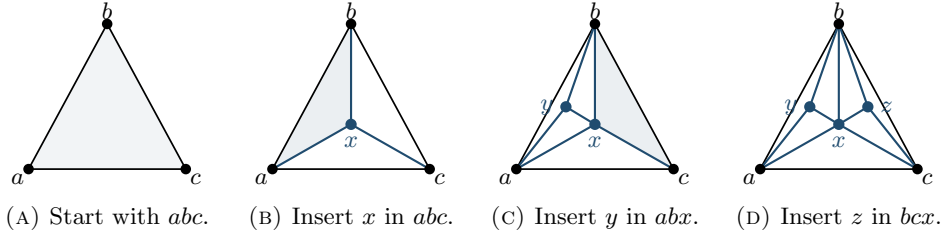

    \centering
    \begin{subfigure}[t]{.235\textwidth}
        \centering\subdivisionpicture{0}
        \caption{Start with $abc$.}
    \end{subfigure}\hfill
    \begin{subfigure}[t]{.235\textwidth}
        \centering\subdivisionpicture{1}
        \caption{Insert $x$ in $abc$.}
    \end{subfigure}\hfill
    \begin{subfigure}[t]{.235\textwidth}
        \centering\subdivisionpicture{2}
        \caption{Insert $y$ in $abx$.}
    \end{subfigure}\hfill
    \begin{subfigure}[t]{.235\textwidth}
        \centering\subdivisionpicture{3}
        \caption{Insert $z$ in $bcx$.}
    \end{subfigure}
    \caption{Obtain $G'$ from $G$ with respect to $b$.}
    \label{fig:subdivisions}
\end{figure}

Consider a nonzero element $\varepsilon\in \W^\dual$, and define
$$
    A_\varepsilon=\{x\in \W \mid \varepsilon(x)=1\}
$$

\begin{lemma}\label{lem:odd}
    Let $G$ be a simple plane triangulation in which every vertex has odd degree.  
    If a characteristic function $\lambda$ on $G$ has at most five values, then $\im\lambda\subset A_\varepsilon$ for some nonzero $\varepsilon\in \W^\dual$.
\end{lemma}

\begin{proof}
Choose distinct $p,q\in \W^\times\setminus\im\lambda$, and put  $t=p+q$.
We claim that $t$ is also not in $\im\lambda$.

Suppose instead that $\lambda(v)=t$ for a vertex $v$.  
List its neighbors in cyclic order as $v_1,\ldots,v_k$, where $k=\deg(v)$.
Consider the quotient map~$\pi\colon \W\longrightarrow \W/\langle t\rangle\cong\Z_2^2$.
For each $i$, the triangle $vv_iv_{i+1}$ is facial, with indices taken modulo $k$.  Hence $t,\lambda(v_i),\lambda(v_{i+1})$ form a basis.
Consequently, $\pi(\lambda(v_i))$ and $\pi(\lambda(v_{i+1}))$ are distinct and nonzero.
One of the three nonzero cosets in this quotient is $p+\langle t\rangle=\{p,q\}$.
Neither of its elements is in $\im \lambda$.  
The cyclic list $\pi(\lambda(v_1)),\ldots,\pi(\lambda(v_k))$ therefore uses only the other two nonzero cosets, and these must alternate.  
This forces $k$ to be even, a contradiction.

Thus $p,q,p+q$ are all missing.  
They are precisely the nonzero vectors of $H=\spn{p,q}$, so $\im\lambda\subset \W\setminus H$.
Taking $\varepsilon$ with $\ker\varepsilon=H$ gives $\W\setminus H=A_\varepsilon$, as required.
\end{proof}

\begin{proof}[Proof of Theorem~\ref{thm:five-four}]
    It is enough to show that $E_5$ implies $E_4$.
    Assume $E_5$, and let~$G$ be an arbitrary simple plane triangulation.
    Apply Lemma~\ref{lem:parity} to obtain a simple plane triangulation $G'$ in which every vertex has odd degree and $G'[V(G)]=G$.
    By $E_5$, the graph $G'$ admits a characteristic function $\lambda'$ with at most five values.  
    Lemma~\ref{lem:odd} gives
    $$
        \im\lambda'\subset A_\varepsilon \qquad\text{for some }0\ne\varepsilon\in \W^\dual.
    $$
    Define $\lambda \coloneq \lambda'|_{V(G)}$.
    All edges of $G$ remain in $G'$, so $\lambda$ assigns distinct values to adjacent vertices of $G$.  
    Thus the three vertices of any facial triangle of $G$ have distinct values in $A_\varepsilon$.  
    Any three such values form a basis, so $\lambda$ is a mod~$2$ characteristic function on $G$ with at most four values.
\end{proof}

\section*{Acknowledgments}
The first author thanks Nikolai Erokhovets for bringing this problem to his attention.

\appendix
\section{On Liu's reduction argument}\label{app:Liu}

The proof of~\cite[Lemma~3.1]{Liu2013} uses a normalized matrix truncation to reduce a \emph{niceness} condition to an earlier polytope in a sequence of cuts.
The local determinant identities in that proof do not, however, control all vertex conditions after truncation.
We give an example in which the matrix before truncation is nice, while the truncated matrix is not.

We use Liu's convention for integral characteristic matrices.
For an $r\times(r-3)$ matrix $S$, let $\Delta_{ijk}(S)$ denote the determinant obtained by deleting rows $i,j,k$.
The characteristic condition is $\Delta_{ijk}(S)=\pm1$ whenever $F_i$, $F_j$, and $F_k$ meet at a vertex.
All rows below are listed in facet order.

Take a triangular prism $P_5$ and write $ijk$ for the vertex $F_i\cap F_j\cap F_k$.
Label its facets so that
$$
    V(P_5)=\{123,125,135,234,245,345\}.
$$
Cut off $123$ and call the new facet $F_6$.
The resulting polytope $P_6$ has vertices
$$
    V(P_6)=\{125,135,234,245,345,126,136,236\}.
$$
Let $U$ be the two-edge path from $125$ through $245$ to $234$.
In Liu's notation, this is the eve-structure with ordered facets~$(F_1,F_2,F_3,F_4,F_5)$.
The vertex cut preserves~$U$, as shown in Figure~\ref{fig:Liu-cut}.

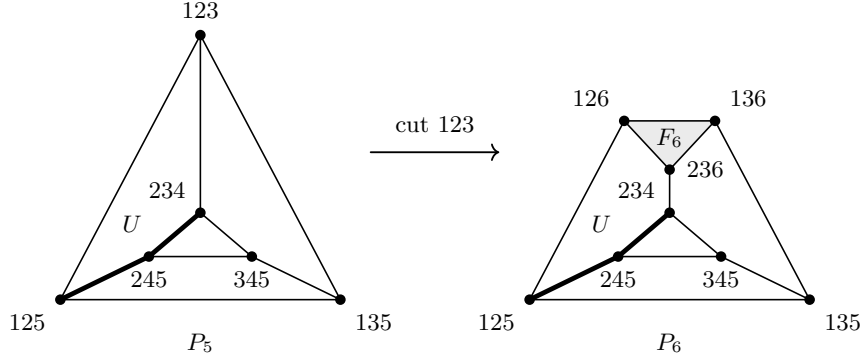
\begin{figure}
\centering
\begin{tikzpicture}[
    font=\small,
    line cap=round,
    line join=round,
    liu edge/.style={line width=.6pt},
    liu path/.style={line width=1.8pt},
    liu vertex/.style={circle,fill=black,inner sep=1.4pt}
]
\foreach \shift/\cut in {-3.1/0,3.1/1}{
    \begin{scope}[xshift=\shift cm]
        \coordinate (a) at (0,1.9);
        \coordinate (b) at (-1.85,-1.6);
        \coordinate (c) at (1.85,-1.6);
        \coordinate (d) at (0,-.45);
        \coordinate (e) at (-.68,-1.03);
        \coordinate (f) at (.68,-1.03);

        \ifnum\cut=0
            \draw[liu edge] (a)--(b) (a)--(c) (a)--(d);
            \node[liu vertex] at (a) {};
            \node[above=3pt] at (a) {$123$};
            \node at (0,-2.17) {$P_5$};
        \else
            \coordinate (l) at (-.60,.765);
            \coordinate (m) at (.60,.765);
            \coordinate (n) at (0,.12);
            \fill[black!8] (l)--(m)--(n)--cycle;
            \draw[liu edge] (l)--(m)--(n)--cycle;
            \draw[liu edge] (b)--(l) (c)--(m) (d)--(n);
            \foreach \p in {l,m,n} \node[liu vertex] at (\p) {};
            \node[above left=2pt] at (l) {$126$};
            \node[above right=2pt] at (m) {$136$};
            \node[right=3pt] at (n) {$236$};
            \node at (0,.54) {$F_6$};
            \node at (0,-2.17) {$P_6$};
        \fi

        \draw[liu edge] (b)--(c) (c)--(f) (d)--(f)--(e)--cycle (b)--(e);
        \draw[liu path] (b)--(e)--(d);
        \foreach \p in {b,c,d,e,f} \node[liu vertex] at (\p) {};
        \node[below left=2pt] at (b) {$125$};
        \node[below right=2pt] at (c) {$135$};
        \node[above left=2pt] at (d) {$234$};
        \node[below=2pt] at (e) {$245$};
        \node[below=2pt] at (f) {$345$};
        \node at (-.90,-.60) {$U$};
    \end{scope}
}
\draw[->,line width=.7pt] (-.85,.35)--(.85,.35) node[midway,above=4pt] {cut $123$};
\end{tikzpicture}
\caption{The vertex cut preserves the thick two-edge path $U$.}
\label{fig:Liu-cut}
\end{figure}

Consider the matrix $S_6$ below.
Deleting its last row and last column gives the matrix $R$:
$$
    S_6=\begin{pmatrix} 1&0&0\\ 0&1&0\\ 0&1&0\\ 1&1&0\\ 1&1&1\\ 0&0&1 \end{pmatrix} \longmapsto R= \begin{pmatrix} 1&0\\ 0&1\\ 0&1\\ 1&1\\ 1&1 \end{pmatrix}.
$$
All eight vertex minors of $S_6$ are $\pm1$, so it is a
characteristic matrix of $P_6$.
Its rows $1,3,6$ form the identity matrix, so it is already
normalized at $245$.
Thus $R$ is precisely the truncation prescribed in the proof
of~\cite[Lemma~3.1]{Liu2013}.
The two local minors used there satisfy
$$
    \Delta_{124}(S_6)=\Delta_{124}(R)=-1 \wideand \Delta_{145}(S_6)=\Delta_{145}(R)=0.
$$
By~\cite[(2.1)]{Liu2013}, $S_6$ is therefore $(2,4)$-nice at $U$.
This means that the prescribed matrix extension is characteristic after cutting out $U$.

However,
$$
    \Delta_{123}(R)= \det\begin{pmatrix}1&1\\1&1\end{pmatrix}=0.
$$
The vertex $123$ lies outside $U$ and survives cutting out $U$ in $P_5$.
Any extension of~$R$ by one column and a final row $(0,0,1)$ still has determinant zero at this vertex, by expansion along the final row.
Thus $R$ is not $(2,4)$-nice at $U$.

This example gives a counterexample to the preservation of niceness claimed for the normalized truncation in the proof of~\cite[Lemma~3.1]{Liu2013}.
This reduction is a key step in the proof of~\cite[Theorem~1.2]{Liu2013}.
Therefore, its failure leaves a gap in that reduction argument.


\end{document}